\documentclass[12pt,english,a4paper]{smfart}

\usepackage[T1]{fontenc}
\usepackage{lmodern}
\usepackage{smfthm}
\usepackage[headings]{fullpage}
\usepackage{amssymb}
\usepackage{tikz-cd}

\newcommand{\Spec}{\operatorname{Spec}}
\newcommand{\W}{(\mathbb{Z}/2\mathbb{Z})^g\rtimes S_g}
\newcommand{\GR}{\operatorname{OGr}(g,V)}
\newcommand{\CC}{\mathbb{C}}
\newcommand{\FF}{\mathcal{F}}
\newcommand{\II}{\mathcal{I}}
\newcommand{\JJ}{\mathcal{J}}
\newcommand{\HH}{\operatorname{H}}
\newcommand{\PP}{\mathbb{P}^{2^g-1}}
\newcommand{\DD}{\operatorname{dim}}

\newcommand{\OO}{\mathcal{O}}

\renewcommand{\phi}{\varphi}

\newcommand{\Y}[1]{%
\begin{tikzpicture}[
    scale=0.18,
    baseline={(0,-0.45)},
    line width=0.3pt
]
\foreach \row [count=\y from 0] in {#1} {
    \foreach \x in {1,...,\row} {
        \filldraw[
            fill=black!10,
            draw=black
        ]
        (\x-1,-\y) rectangle (\x,-\y-1);
    }
}
\end{tikzpicture}%
}

\title{Number of equations defining the Kummer varieties of hyperelliptic curves}

\author{Nathan Partouche}
\email{nathan.partouche@ens-paris-saclay.fr}

\begin{document}

\begin{abstract}
 The Jacobian of an hyperelliptic curve $C$ of genus $g$ over $\CC$ is equipped with a symmetric theta divisor $\Theta$ that defines a principal polarization. The line bundle $\OO(2\Theta)$ defines an embedding of the Kummer variety $K$ inside the projective space $\PP$, which factorizes inside a certain orthogonal Grassmannian. We describe a minimal free resolution of $K$ inside this Grassmannian using geometric representation theory.
As a consequence, we obtain a formula for the number of quartic equations defining $K$ in the projective space $\PP$.

\end{abstract}

\maketitle

\tableofcontents

\setlength{\baselineskip}{18pt}

\section*{Introduction}
\label{intro}

Let $(A,\lambda:A\rightarrow A^{\vee}) $ be a principally polarized indecomposable $g$-dimensional abelian variety over $\CC$. Let $L$ be a symmetric line bundle inducing the polarization $\lambda$. Then $L^{\otimes 2}$ is independent of the choice of $L$, and choosing a basis of global sections of $L^{\otimes 2}$ defines a morphism $A\rightarrow\PP$ which factors through an embedding of the Kummer variety $K = J/\langle \pm 1\rangle$. 

The line bundle $L^{\otimes2}$ on $A$ descends to a line bundle $M$ on the Kummer variety $K$. In this paper, we will study the equations defining the Kummer variety in the projective space $\PP$ regarding the closed immersion $i$.

Khaled in \cite[Theorem 3.2]{Kha93} showed that an arbitrary Kummer variety has its ideal sheaf of equations inside $\PP$ generated by quadrics, cubics, and quartics.

Bert van Geemen in \cite[Section 2.2]{Gee16} showed that generically there is no quadric vanishing on $K$ and that the map
\begin{align*}
    \HH^0(K,M)\otimes\HH^0(K,M)\to\HH^0(K,M^{\otimes2})
\end{align*}
is surjective.
The genericity condition is the non-vanishing of certain Theta nullwerte.

Ryuji Sasaki in \cite[Corollary 1.5]{Sas81} showed that, under the hypothesis that the multiplication map $\HH^0(K,M)\otimes\HH^0(K,M)\to\HH^0(K,M^{\otimes2})$ is surjective, the ample line bundle $M$ is projectively normal.
Combining these two results, we conclude that for a generic principally polarized abelian variety, no quadrics vanish on $K$, and that the ample line bundle $M$ on $K$ is normally generated.

We will now focus on a certain family of principally polarized indecomposable abelian varieties over $\CC$.
Let $C$ be an hyperelliptic curve of genus $g$ over $\CC$. The Jacobian $J_C$ of the hyperelliptic curve $C$ is equipped with a symmetric theta divisor $\Theta$ that defines a principal polarization. Each such hyperelliptic curve $C$ of this form gives rise to a principally polarized indecomposable abelian variety $(J_C,\OO(\Theta))$. For dimension $g$ equal to $1$ or $2$, every indecomposable $(A,\lambda)$ is of this form, and in general the locus of such hyperelliptic Jacobians defines a closed subscheme of dimension $2g-1$ inside the coarse moduli space of principally polarized abelian varieties (which has dimension $\frac{g(g+1)}{2}$). For dimension $g$ at least $3$, the multiplication map 
\begin{align*}
\HH^0(K,M)\otimes\HH^0(K,M)\to\HH^0(K,M^{\otimes2})
\end{align*}
is not surjective for hyperelliptic Jacobians. The family of hyperelliptic Jacobians does not fit in the generic picture. 

The Kummer variety with the line bundle $M$ (whose pullback is $\OO(2\Theta)$ on $J_C$) is not projectively normal, and there is a quadric vanishing on the Kummer variety, for $g\geq3$ as recorded in \cite[Section 2.9.3]{Kha95}.
This is in sharp contrast with the generic situation.

\subsection*{Main results}
Bert van Geemen \cite[Theorem 2]{Gee88} proved that the set of quadrics vanishing on the hyperelliptic Kummer variety $K$ defines an orthogonal Grassmannian $\operatorname{OGr}(g,2g+1)$. Inside this orthogonal Grassmannian, we have found the minimal free resolution of $K$. 

Denote by $\II$ the ideal sheaf of the closed immersion $i:K\to\PP$.
Jef Laga and Jack Thorne showed in \cite[Proposition 3.16]{LT25} that the Kummer variety $K$ is defined scheme-theoretically by quadratic and quartic equations. Let $Q$ be the vector space of quartics vanishing on $K$. If we take a quadratic equation vanishing on the Kummer variety and we multiply it with an arbitrary quadratic equation, we obtain a quartic vanishing on the Kummer variety. Let $Q_0$ be the subspace generated by quartics arising from quadrics.
Using the minimal free resolution of $K$ inside $\operatorname{OGr}(g,2g+1)$, one obtains the following theorem.

\begin{enonce*}{Theorem}
\label{theo}
For all $p\geq1$, the cohomology group $\HH^p(\PP,\II(4))$ vanishes, and there are
\begin{align*}
    \dim Q=\binom{2^g+3}{4}-\frac{8^g+2^g}{2}
\end{align*} linearly independent quartics defining the Kummer variety and \begin{align*}
    \dim Q/Q_0=\frac{2^g}{2g+1}\binom{2g+1}{g}\prod_{1\leq i < j\leq g}\frac{2g+4-i-j}{2g+2-i-j}-\frac{8^g+2^g}{2}
\end{align*} linearly independent quartics modulo quartics arising from quadrics.
\end{enonce*}

The formula for the number of linearly independent quartics defining the Kummer variety inside $\PP$ is also true in the generic case. This comes from the fact that the embedding $K\rightarrow\PP$ is projectively normal.

The geometry of the Kummer variety of low-genus hyperelliptic curves was worked on explicitly by mathematicians. A consequence of their work was that for low-genus hyperelliptic curves, the number of quadrics and quartics defining the Kummer variety was obtained. See the work of Michael Stoll in \cite[Theorem 2.5]{Sto17} for genus equals to $3$. The main theorem of the article gives a formula that works for every genus $g$. 

\subsection*{Strategy}
The closed immersion $K\to\PP$ can be factored, using the specific structure arising from the hyperelliptic curve, into two closed immersions $K\to\operatorname{OGr}(g,2g+1)$ and $\operatorname{OGr}(g,2g+1)\to\PP$. Using the equivariant geometry of the orthogonal Grassmannian, we obtain a minimal free resolution of the ideal sheaf of $K\to\operatorname{OGr}(g,2g+1)$. This minimal free resolution is composed of sheaves for which the cohomology can be computed with the Borel-Weil-Bott theorem. The spectral sequence associated with the hypercohomology will then prove the theorem.
This minimal free resolution is the cornerstone of the article. It encodes a substantial amount of information, including the formula for the number of quadrics, but its scope extends considerably beyond this result.
The idea of studying geometric objects related to abelian varieties using geometric representation theory also appears in \cite{GSW13}.

\subsection*{Organization}
In Section \ref{Technical tools}, we will review the tools required to obtain the proof of the theorem. In Section \ref{The Kummer embedding}, we will explain how the results of Laga and Thorne in \cite{LT25} give a minimal free resolution consisting of sheaves that will enable us to show the vanishing of $\HH^1(\PP,\II(4))$ using combinatorics. The last section, Section \ref{Combinatorics}, will contain the proof of the theorem.

\subsection*{Acknowledgements}
The research was carried out while the author was doing his master's internship at the University of Cambridge under the supervision of Jef Laga, to whom I am grateful for his guidance.

No artificial intelligence tools were used in the preparation of this work.
\section{Technical tools}\label{Technical tools}
\subsection{Equivariant sheaves}
We will briefly recall the definition of equivariant sheaves and prove a lemma. Let $X$ be a scheme over $\CC$, and let $G$ be a group scheme over $\CC$ acting on $X$. We denote by $p:G\times X\to X$ the projection, by $a:G\times X\to X$ the action, by $m:G\times G\to G$ the multiplication and by $p_{23}:G\times G\times X\to G\times X$ the projection on the last two factors.

\begin{defi}
    A quasi-coherent sheaf $\FF$ is called $G$-equivariant if there exist an isomorphism $\phi_\FF:a^*\FF\to p^*\FF$ satisfying the cocycle condition $(m\times1)^*\phi=p_{23}^*\phi\circ(1\times a)^*\phi$

    A morphism of quasi-coherent sheaves $u:\FF\to\mathcal{G}$ between $G$-equivariant sheaves is $G$-equivariant if $\phi_\mathcal{G}\circ a^*u=p^*u\circ\phi_\FF$.
\end{defi}

Now suppose we have another scheme $Y$, an open cover $\{U_i\}$ of $Y$, maps $f_i:U_i\to X$, and morphisms $g_{ij}:U_i\cap U_j\to G$ satisfying $a\circ(g_{ij},f_{j_{|U_i\cap U_j}})=f_{i_{|U_i\cap U_j}}$ and $m\circ(g_{ij_{|U_i\cap U_j\cap U_k}},g_{jk_{|U_i\cap U_j\cap U_k}})=g_{ik_{|U_i\cap U_j\cap U_k}}$.
 
Let $u:\FF\to\mathcal{G}$ be a morphism of $G$-equivariant sheaves on $X$; then we can glue $f_i^*\FF$, $f_i^*\mathcal{G}$ and $f_i^*u$ in a coherent way to obtain a morphism of quasi-coherent sheaves on $Y$, as we now explain.

We pull back the isomorphism $\phi_\FF$ defining the $G$-equivariant structure by $(g_{ij},f_j)$ to obtain an isomorphism $\psi_{ij}:f_i^*\FF_{|U_i\cap U_j}\to f_j^*\FF_{|U_i\cap U_j}$. 

\begin{lemm}
    \label{equiv}
    The $\psi_{ij}$ satisfies the Zariski cocycle condition. 
\end{lemm}

\begin{proof}
    It suffices to pull back the cocycle condition $(m\times1)^*\phi_\FF=p_{23}^*\phi_\FF\circ(1\times a)^*\phi_\FF$ by $(g_{ij_{|U_i\cap U_j\cap U_k}},g_{jk_{|U_i\cap U_j\cap U_k}},f_{k_{|U_i\cap U_j\cap U_k}})$.
    We obtain the desired Zariski cocycle condition $\psi_{jk_{|U_i\cap U_j\cap U_k}}\circ\psi_{ij_{{|U_i\cap U_j\cap U_k}}}=\psi_{ik_{{|U_i\cap U_j\cap U_k}}}$. 
\end{proof}

 We denote by $\psi'_{ij}$ the isomorphism in the case of $\mathcal{G}$, pulling back $\phi_\mathcal{G}\circ a^*u=p^*u\circ\phi_\FF$ by $(g_{ij},f_{j_{|U_i\cap U_j}})$ gives $\psi'_{ij}\circ f_i^*u_{| U_i\cap U_j}=f_j^*u_{|U_i\cap U_j}\circ\psi_{ij}$. We can then glue these $f_i^*u$.

Conceptually, the above data determines a map $g: Y\rightarrow[X/G]$, then $\FF$, $\mathcal{G}$ and $u$ descend to the quotient stack $[X/G]$. We are simply pulling these objects via the map $g$.

\subsection{Commutative algebra}
For the formalism of the Schur functor, we recommend \cite[Section 2]{Wey03} for a detailed treatment or \cite[Section 2.1]{GSW13} for a brief description of the definition and properties.

Let $X$ a scheme and $\mathcal{F}$ a locally free sheaf of rank $g$ on $X$.
Let $\lambda=(\lambda_1,\dots,\lambda_g)\in\mathbb{Z}^g$ such that $\lambda_1\geq\lambda_2\geq\dots\geq\lambda_{g-1}\geq\lambda_g$. If $\lambda_g$ is non negative, then we say that $\lambda$ is "nice".
We denote by $L_\lambda\FF$ (resp. by $K_\lambda\FF$) the Schur (resp. Weyl) functor associated with $\lambda$ applied to $\FF$. Let $\lambda'\in\mathbb{Z}^g$ denote the transpose of $\lambda$, that is, $\lambda'_i$ is the number of components of $\lambda$ that are at least $i$.
If the nice vector $\lambda$ is equal to its transpose $\lambda'$, then the Schur functor $L_\lambda\FF$ is isomorphic to the Weyl functor $K_\lambda\FF$.

The important point is that Schur functors are functorial and commute with base change. Hence, if we have a group scheme $G$ acting on the locally free sheaf $\FF$, then there is a natural $G$-equivariant structure on $L_\lambda\FF$. If we consider a Schur functor associated to a nice vector $\lambda=(\lambda_1,\dots,\lambda_g)$ with $\lambda_g$ non-zero, and apply it to a locally free sheaf of rank strictly less than $g$, then we obtain the zero sheaf. If $\lambda_g$ is zero, then $L_\lambda\FF=L_{(\lambda_1,\dots,\lambda_{g-1})}\FF$ and $K_\lambda\FF=K_{(\lambda_1,\dots,\lambda_{g-1})}\FF$.

The locally free sheaf $K_{(1,\dots,1)}\FF$ corresponds to the line bundle $\operatorname{det}\FF$, and one rule of the Schur calculus says that the tensor production $K_\lambda\FF\otimes\operatorname{det}\FF$ is equal to $K_{\lambda+(1,\dots,1)}\FF$.

Another important result is the Eagon-Northcott perfection theorem, giving conditions ensuring that the pullback of a minimal free resolution remains a minimal free resolution.
Let $I$ be an ideal of $\CC[X_i]_{1\leq i\leq r}$ such that $\CC[X_i]_{1\leq i\leq r}/I$ is Cohen-Macaulay. Let $F_\bullet$ be a $\CC[X_i]_{1\leq i\leq r}$-linear minimal free resolution of $I$.
Let $A$ be a Noetherian $\CC[X_i]_{1\leq i\leq r}$-algebra.

\begin{enonce*}{Theorem 1.3 (Eagon-Northcott perfection)}
    Suppose that $\DD A-\DD A/IA=r-\DD\CC[X_i]_{1\leq i\leq r}/I$, then $F_\bullet\otimes A$ is a minimal free resolution of $I\otimes A$.
\end{enonce*}
\begin{proof}
    See \cite[Theorem 3.5]{BV88}
\end{proof}

\subsection{Borel-Weil-Bott}
We will simply state the theorem for a particular flag variety. Jerzy Weyman's book \cite[Section 4]{Wey03} is a nice reference for the Borel-Weil-Bott theorem.

Let $(V,\psi)$ be a $2g+1$-dimensional $\CC$-vector space equipped with a non-degenerate symmetric bilinear form. 
Let $\operatorname{Gr}(g,V)$ be the usual Grassmannian of $g$-dimensional subspaces of $V$ over $\CC$ on which the group scheme $\operatorname{GL}(V)$ acts. Let $\FF_0$ be the universal $g$-dimensional subspace, which is a $g$-dimensional sub-vector bundle of $V\otimes\OO_{\operatorname{Gr}(g,V)}$. This vector bundle has a structure of $\operatorname{GL}(V)$-equivariant vector bundle. The locus where the map $\psi\otimes\OO_{\operatorname{Gr}(g,V)}:(V\otimes\OO_{\operatorname{Gr}(g,V)})\times (V\otimes\OO_{\operatorname{Gr}(g,V)})\to\OO_{\operatorname{Gr}(g,V)}$ vanishes on $\FF_0\times\FF_0$ is called the orthogonal Grassmannian $\GR$. This locus is stabilized by the group scheme $\operatorname{SO}(V)$. We  denote by $\FF$ the restriction of the vector bundle $\FF_0$ on $\GR$, which is endowed with a $\operatorname{SO}(V)$-equivariant structure induced by the $\operatorname{GL}(V)$-equivariant structure of $\FF_0$.

The group $\W$ acts on $\mathbb{R}^g$ by changing signs and permuting the components. Let $\rho$ denote the vector $(\frac{2g-1}{2},\frac{2g-3}{2},\dots,\frac{3}{2},\frac{1}{2})$. For $\sigma\in\W$ and $\lambda\in\mathbb{R}^g$, we will denote by $\sigma\cdot\lambda=\sigma(\lambda+\rho)-\rho$ the twisted action of $\W$ on $\mathbb{R}^g$. For $\sigma\in\W$, its length $l(\sigma)$ is defined as the sum of the minimal number of transpositions of the form $(i,i+1)$ needed to write the permutation of the $S_g$-component with the number of components equal to $1$ in the $(\mathbb{Z}/2\mathbb{Z})^g$-component. 

For a nice vector $\alpha\in\mathbb{Z}^g$, we denote by $V_\alpha$ the irreducible representation of $\operatorname{SO}(V)$ of highest weight $\alpha$.
Let $\beta=(\beta_1,\dots,\beta_g)\in\mathbb{Z}^g$ such that $\beta_1\geq\beta_2\geq\dots\geq\beta_{g-1}\geq\beta_g$.
Let $\gamma\in\mathbb{Z}^g$ be the vector $(-\beta_g,\dots,-\beta_1)$. The Borel-Weil-Bott theorem for $\GR$ is the following dichotomy allowing the computation of the cohomology of the $\operatorname{SO}(V)$-equivariant vector bundle $K_\beta\FF$.

\begin{enonce*}{Theorem 1.4 (Borel-Weil-Bott)}
    \label{BWB}
    One of two mutually exclusive possibilities occurs:
    \begin{enumerate}
        \item There exists $\sigma\in\W$ such that $\sigma\cdot\gamma=\gamma$, then for all $p\geq0$, the cohomology group $\HH^p(\GR,K_\beta\FF)$ vanishes.
        \item There exists a unique $\sigma\in\W$ such that $\sigma\cdot\gamma$ is nice, then for all $p\neq l(\sigma)$, the cohomology group $\HH^p(\GR,K_\beta\FF)$ vanishes and $\HH^{l(\sigma)}(\GR,K_\beta\FF)=V_{\sigma\cdot\gamma}$.
    \end{enumerate}
\end{enonce*}

What will be called the Borel-Weil-Bott procedure is the algorithm where we begin with a nice vector $\beta$, then do all the transformations, and try to find if we are in the first case or the second case of theorem 1.4.
The first step is to consider the vector $\gamma=(-\beta_g,\dots,-\beta_1)$, the second step is to compute $\gamma+\rho$, and the last step is to find $\sigma\in\W$ such that $\sigma\cdot\gamma$ is nice. Finding $\sigma\in\W$ such that $\sigma\cdot\gamma$ is nice is equivalent to find $\sigma\in\W$ such that $\sigma(\gamma+\rho)$ is nice with distinct positive components.

\subsection*{Examples}
It takes the following form for $\beta=(2,2)$.
\begin{align}
    (2,2)\rightarrow(-2,-2)\rightarrow(-\frac{1}{2},-\frac{3}{2})\text{, } \sigma=((1,1),(1,2)) \text{ and } l(\sigma)=3
\end{align} 

For $\beta=(1,1)$, it takes the following form.
\begin{align}
    (1,1)\rightarrow(-1,-1)\rightarrow(\frac{1}{2},-\frac{1}{2})\text{}
\end{align}
There is no $\sigma\in\W$ such that $\sigma\cdot\gamma$ is nice. This is because at the end of the second step, two components are equal up to the sign.

\begin{align}
    (0,0)\rightarrow(0,0)\rightarrow(\frac{3}{2},\frac{1}{2})\text{, } \sigma=((1,1),\operatorname{id}) \text{ and } l(\sigma)=0
\end{align}
\begin{align}
    (2,2,0)\rightarrow(0,-2,-2)\rightarrow(\frac{5}{2},-\frac{1}{2},-\frac{3}{2})\text{, }\sigma=((0,1,1),(2,3))\text{ and }l(\sigma)=3
\end{align}
\begin{align}
    (3,2,1)\rightarrow(-1,-2,-3)\rightarrow(\frac{3}{2},-\frac{1}{2},-\frac{5}{2})\text{, }\sigma=((0,1,1),(1,2)\circ(2,3))\text{ and }l(\sigma)=4
\end{align}
\begin{align}
    (3,3,2)\rightarrow(-2,-3,-3)\rightarrow(\frac{1}{2},-\frac{3}{2},-\frac{5}{2})\text{, }\sigma=((0,1,1),(2,3)\circ(1,2)\circ(2,3))\text{ and }l(\sigma)=5
\end{align}
\begin{align}
    (0,0,-2)\rightarrow(2,0,0)\rightarrow(\frac{9}{2},\frac{3}{2},\frac{1}{2})\text{, }\sigma=((0,0,0),\operatorname{id})\text{ and }l(\sigma)=0
\end{align}
\begin{align}
    (1,0,-1)\rightarrow(1,0,-1)\rightarrow(\frac{7}{2},\frac{3}{2},-\frac{1}{2})\text{, }\sigma=((0,0,1),\operatorname{id})\text{ and }l(\sigma)=1
\end{align}
\begin{align}
    (1,1,0)\rightarrow(0,-1,-1)\rightarrow(\frac{5}{2},\frac{1}{2},-\frac{1}{2})
\end{align}

We see that for a nice vector $(\beta_1,\dots,\beta_g)\in\mathbb{N}^g$, doing the Borel-Weil-Bott procedure on $(\beta_1,\dots,\beta_g)\in\mathbb{N}^g$ and $(\beta_1,\dots,\beta_g,0)\in\mathbb{N}^{g+1}$ will yield the same length.
These examples will serve during the article.

\subsection{Symmetric determinantal varieties}

Let us consider the $\CC$-scheme $s:X\rightarrow\CC$ of $g\times g$ symmetric matrices, which is the scheme $\Spec(\CC[X_{i,j}]_{1\leq i,j\leq g}/(X_{i,j}-X_{i,j})_{1\leq i\leq j\leq g})$. The group scheme $\operatorname{GL}_g$ acts on $X$ by congruence; at the level of points, we have $g\cdot M=gMg^t$ for $g\in\operatorname{GL}_g(T)$, $M\in X(T)$, and $T$ a $\CC$-scheme.
There is the universal symmetric matrix $(X_{i,j})_{1\leq i,j\leq g}\in M_{g\times g}(\Gamma(X,\OO_{X}))$. The locus where the universal matrix is of rank at most $1$ defines a closed subscheme $Z$ of $X$ called the symmetric determinantal variety of rank at most $1$. The ideal sheaf of $Z$ in $X$ will be denoted by $\JJ$.

The symmetric determinantal variety of rank at most $1$ is well studied; the coordinate ring of $Z$ is Cohen-Macaulay (see \cite[Theorem 6.3.1]{Wey03}). We also have an explicit description of a minimal free resolution of the ideal sheaf in terms of $\operatorname{GL}_g$-equivariant objects and $\operatorname{GL}_g$-equivariant differentials. 

Let $u\in\mathbb{N}$ be an integer, and $\alpha=(\alpha_1,\dots,\alpha_{2u})\in\mathbb{Z}^{2u}$ be nice. Let $|\alpha|$ be the sum of the components of $\alpha$. Let $\alpha'\in\mathbb{Z}^{2u}$ be the transpose of $\alpha$.
We define $\lambda(\alpha,u)$ to be the following nice vector $(2u+\alpha_1,\dots,2u+\alpha_{2u},\alpha'_1,\dots,\alpha'_{2u})\in\mathbb{Z}^{4u}$. For $i\geq1$, we denote by $I_{i}$ the set $\{\text{nice } \alpha|i=|\alpha|+2u^2-u\}$. Let $E=\CC^g$, so that $\operatorname{GL}_g$ acts on $E$. It defines a $\operatorname{GL}_g$-equivariant locally free sheave $\widetilde{E}$ of rank $g$ on $\operatorname{Spec}(\CC)$.
\begin{enonce*}{Theorem 1.5 (Resolution of $\JJ$)}
    For $i\geq1$, the $i$-th term $G_i$ of the minimal free resolution of $\JJ$ is 
    \begin{align*}
    \bigoplus_{u\geq 0}\bigoplus_{\alpha\in I_{i,u}}s^*(\widetilde{L_{\lambda(\alpha,u)}E}).
    \end{align*}
    For the $\operatorname{GL}_g$-equivariant structure induced by the $\operatorname{GL}_g$-action on $E$ and functoriality of $L_\lambda$, the differentials are $\operatorname{GL}_g$-equivariant.
\end{enonce*}
\begin{proof}
    See \cite[Theorem 6.3.1]{Wey03}
\end{proof}

\section{The Kummer embedding}\label{The Kummer embedding}
The embedding of the Kummer variety $i:K\to\PP$ (arising from a hyperelliptic curve) is easier to understand than in the general case, thanks to the results of Laga and Thorne.  They construct a $2g+1$-dimensional $\CC$-vector space $V$ equipped with a non-degenerate symmetric bilinear form. They showed in \cite[Section 3]{LT25} that $i$ factors through two closed immersions: $i':K\to\GR$ and $\Sigma:\GR\to\PP$. The second closed immersion is the spinor embedding, well studied by Henri Cartan and Claude Chevalley.

Laga and Thorne also construct \cite[Theorem 3.13]{LT25} a symmetric bilinear form $\psi:\FF\times\FF\to\OO_{\GR}$ such that $i'(K)$ is the locus where $\psi$ has rank at most $1$. Another important result is that the pullback by $\Sigma$ of $\OO(2)$ is $(\det\FF)^{\vee}$. One rule of Schur calculus says that the tensor product $K_\lambda\FF\otimes\det\FF$ is equal to $K_{\lambda+(1,\dots,1)}\FF$.

Let $\II'$ be the ideal sheaf associated with the closed immersion $i'$. The Kummer variety is defined scheme-theoretically by quartics \cite[Proposition 3.16]{LT25} inside the orthogonal Grassmannian $\Sigma(\GR)$ and the quadrics vanishing on the Kummer variety are the same as the quadrics vanishing on the orthogonal Grassmannian \cite[Proposition 3.17]{LT25}. 

The fact that the quadrics vanishing on the Kummer variety defines an orthogonal Grassmannian was already known before the paper of Laga and Thorne \cite{LT25}; see, for example, the following paper of Bert van Geemen \cite[Theorem 2]{Gee88}. The paper of Laga and Thorne \cite{LT25} mainly deals with making this much more systematic and explicit and applying it to arithmetic contexts.

Let $U$ be an affine open subset of $\GR$ that trivializes $\FF$. We then have an isomorphism $\phi_U:\FF_{|U}\to\OO_U^g$. Under this choice of coordinates, the matrix defined by $\psi$ gives rise to a morphism $f_U:U\to X$ such that the following diagram is cartesian.
\[\begin{tikzcd}
	{U\cap i'(K)} & Z \\
	U & {X}
	\arrow[from=1-1, to=1-2]
	\arrow[from=1-1, to=2-1]
	\arrow[from=1-2, to=2-2]
	\arrow["f_{U}",from=2-1, to=2-2]
\end{tikzcd}\].

The Kummer variety $K$ and this orthogonal Grassmannian $\GR$ are integral schemes of finite type over $\CC$. Their dimension can be computed via the dimension of any nonempty open subset. 

\begin{lemm}
    The dimensions of $X$ and $\GR$, as well as $Z$ and $K$, are the same. 
\end{lemm}
\begin{proof}
    The Jacobian $J_C$ and the Kummer variety $K$ are integral schemes of finite type over $\CC$. As the map $J\rightarrow K$ is finite, the dimension of $K$ is $g$ thanks to a consequence of the Going-Up theorem.
    The scheme $X$ is $\mathbb{A}_\CC^{\frac{g(g+1)}{2}}$, his dimension is then $\frac{g(g+1)}{2}$.
   
    The two other dimensions will be computed with an incidence argument.
    For the dimension of the symmetric determinantal variety, we have the following diagram for every $\CC$-scheme $T$.
    \begin{align*}
        Z(T)\leftarrow \{(M,l)\in X(T)\times\mathbb{P}^{g-1}(T)/\operatorname{Im}M\subset l\}\rightarrow\mathbb{P}^{g-1}(T)
    \end{align*}
    where the two maps are surjective morphisms between integral schemes over $\CC$. The fiber over a non-zero symmetric matrix is of dimension $0$, and the fiber over a line of $\mathbb{P}^{g-1}$ is of dimension $1$. Using the fiber dimension theorem, we obtain that the dimension of $Z$ is $g$.
    For the dimension of the orthogonal Grassmannian $\operatorname{OGr}(g,V)$, it is explained in \cite[Chapter 6, Section 1]{GH78}.
    
\end{proof}

Applying the Eagon Northcott perfection theorem, we obtain a minimal free resolution of $\II'_{|U}$. 

The choice of a different isomorphism $\phi'_U:\FF_{|U}\to\OO_U^g$ gives rise to a morphism $g_U:U\to\operatorname{GL}_g$ such that $f_U'=a\circ(g_U,f_U)$. We are therefore in the situation of the lemma \ref{equiv}, and we can glue the minimal free resolution of the different $\II'_U$ to obtain a minimal free resolution of $\II'$.

\begin{lemm}\label{resfofI'}
    For $i\geq1$, the $i$-term $G_i$ of the minimal free resolution of $\II'$ is  
    \begin{align*}
    \bigoplus_{u\geq 0}\bigoplus_{\alpha\in I_{i,u}}K_{\lambda(\alpha,u)}\FF.
    \end{align*}
\end{lemm}
\begin{proof}
We need to investigate the gluing data coming from the lemma \ref{equiv}.
The isomorphism of the $\operatorname{GL}_g$-equivariant structure of $\widetilde{L_\lambda E}$ is the pullback by the structure map $s:X\to\Spec(\CC)$ of the isomorphism of the $\operatorname{GL}_g$-equivariant structure of $L_\lambda E$. 
Let $\phi_{L_\lambda E}$ denote the isomorphism defining the $\operatorname{GL}_g$-equivariant structure of $L_\lambda E$. The isomorphism between $f_U'^*(\widetilde{L_\lambda E})$ and $f_U^*(\widetilde{L_\lambda E})$ is $g_U^*\phi_{L_\lambda E}$.

We will show that $g_U^*\phi_{L_\lambda E}=L_\lambda((g_U^\#X_{ij})_{1\leq i,j\leq g})$, it suffices to check the equality at the closed point of $U$. Let $x:\Spec(\CC)\to U$ be a closed point, the morphism $x^*g_U^*\phi_{L_\lambda E}$ is the morphism $L_\lambda(x^\#g_U^\#X_{ij})_{1\leq i,j\leq g}:L_\lambda E\to L_\lambda E$ by definition of the action of $\operatorname{GL}_g(\CC)$ on $L_\lambda E$. And the morphism $x^*L_\lambda((g_U^\#X_{ij})_{1\leq i,j\leq g})$ is also equal to $L_\lambda(x^\#g_U^\#X_{ij})_{1\leq i,j\leq g}:L_\lambda E\to L_\lambda E$ thanks to the compatibility of Schur functors with base change.
\end{proof}

However, the differentials of this complex are not $\operatorname{SO}(V)$-equivariant. The method from representation theory will not be useful to the computation of the differentials.
Thanks to the lemma \ref{resfofI'}, the Borel-Weil-Bott theorem, and  the fact that $\det\FF$ is $\OO_{\GR}(-2)$, we can now use the spectral sequence associated with hypercohomology to obtain results on the cohomology of the $\II'(2n)$ for $n\in\mathbb{Z}$. We will only consider the cases $0\leq n\leq2$.

\subsection*{Examples}
Here are the partitions arising for small values of $i$.
\[
\begin{array}{c@{\qquad}c@{\qquad}c@{\qquad}c}
i=1 & i=2 & i=3 & i=4 \\[2mm]
\Y{2,2} & \Y{3,2,1} &
\begin{array}{c@{\hspace{3mm}}c}
\Y{4,2,1,1} & \Y{3,3,2}
\end{array} &
\begin{array}{c@{\hspace{3mm}}c}
\Y{5,2,1,1,1} & \Y{4,3,2,1}
\end{array}
\end{array}
\]

For genus $g$ equal to $1$, all the terms of the minimal free resolution vanish; this means that the morphism $K\rightarrow\mathbb{P}$ is an isomorphism.

For genus $g$ equal to $2$, the ideal sheaf is isomorphic to $K_{(2,2)}\FF$, then $\II'(2)$ is isomorphic to $K_{(1,1)}\FF$, which has no cohomology (example $(2)$ of section $1.3$), and $\II'(4)$ is isomorphic to $K_{(0,0)}\FF$, which has cohomology only in degree 0 (example $(3)$ of section $1.3$). The dimension of $\operatorname{H}^0(\operatorname{OGr}(2,V),\II'(4))$ is one. 

The morphism $\operatorname{OGr}(2,V)\rightarrow\mathbb{P}^3$ is an isomorphism (by dimension the image of the orthogonal Grassmannian inside $\mathbb{P}^3$ is $\mathbb{P}^3$ and the scheme $\mathbb{P}^3$ is reduced). The Kummer variety viewed in the projective space $\mathbb{P}^3$ is defined as a quartic surface; this is the so-called Kummer surface.

For genus $g$ equal to $3$, the minimal free resolution of $\II'$ is
\begin{align*}
    0\rightarrow K_{(3,3,2)}\FF\rightarrow K_{(3,2,1)}\FF\rightarrow K_{(2,2,0)}\FF\rightarrow\II'\rightarrow0
\end{align*} and the minimal free resolution of $\II'(4)$ is
\begin{align*}
    0\rightarrow K_{(1,1,0)}\FF\rightarrow K_{(1,0,-1)}\FF\rightarrow K_{(0,0,-2)}\FF\rightarrow\II'(4)\rightarrow0.
\end{align*}
The locally free sheaf $K_{(1,1,0)}\FF$ has no cohomology (example $(9)$ of section $1.3$), so we obtain the following short exact sequence 
\begin{align*}
    0\rightarrow\operatorname{H}^0(\operatorname{OGr}(g,V),K_{(0,0,-2)}\FF)\rightarrow\operatorname{H}^0(\operatorname{OGr}(g,V),\II'(4))\rightarrow\operatorname{H}^1(\operatorname{OGr}(g,V),K_{(1,0,-1)}\FF)\rightarrow0
\end{align*}
So using the Weyl dimension formula \cite[Section 24.3]{Hum72}, we obtain that the dimension of $\operatorname{H}^0(\operatorname{OGr}(g,V),K_{(0,0,-2)}\FF)$ is $27$ and the dimension of $\operatorname{H}^1(\operatorname{OGr}(g,V),K_{(-1,0,-1)}\FF)$ is $7$. The dimension of $\operatorname{H}^0(\operatorname{OGr}(g,V),\II'(4))$ is then $34$, which is in accordance with the formula and Stoll's result \cite[Theorem 2.5]{Sto17}.

\section{Combinatorics}\label{Combinatorics}
The final goal of this section is to show that for $p\geq1$, the cohomology group $\HH^p(\GR,\II'(4))$ vanishes. The proof of the next lemma will serve as a toy example for the more involved proof of lemma \ref{comb2}. Let $\alpha=(\alpha_1,\dots,\alpha_g)$ be a nice vector and $u\geq0$.

\begin{lemm}\label{comb1}
    The length of the required $\sigma\in\W$ in the Borel-Weil-Bott procedure applied to $K_{\lambda(\alpha,u)}\FF$ is $|\alpha|+2u^2+u$.
\end{lemm}
\begin{proof}
    We denote by $x_1>\dots>x_r$ the different values taken by $\alpha'=(\alpha_1,\dots,\alpha_g,\dots,0)\in\mathbb{N}^{2u}$.
    
    For $1\leq i\leq r$, we will denote by $\epsilon_i=\operatorname{card}\{j|\alpha'_j=x_i\}$ the number of $\alpha'_i$ equal to the $i$-th largest value taken by the components of $\alpha'$.
    
    For $1\leq i\leq r-1$, we will also denote by $\delta_i=x_i-x_{i+1}$ the difference between the $i$-th largest value taken by the components of $\alpha'$ and the $i+1$-th largest value taken by the components of $\alpha'$. We denote by $\delta_r=x_r$.
    
    We see that the sum of all the epsilon is equal to $2u$ and that the sum of all the delta is equal to $\alpha_1$.

    By definition of the epsilons and the deltas, we have that the partition $\lambda(\alpha,u)$ is of the following form 
\begin{align*}(\overbrace{2u+\sum_{i=1}^r\delta_i,\dots,2u+\sum_{i=1}^r\delta_i}^{\epsilon_1\text{ times}},2u+\sum_{i=2}^r\delta_i,\dots,\epsilon_1+\epsilon_2,\overbrace{\epsilon_1,\dots,\epsilon_1}^{\delta_1\text{ times}})
\end{align*}

We now apply the Borel-Weil-Bott procedure.

The first step is to consider the following vector 
\begin{align*}
    (-\epsilon_1,\dots,-\epsilon_1,-\epsilon_1-\epsilon_2,\dots,-2u-\sum_{i=2}^r\delta_i,-2u-\sum_{i=1}^r\delta_i,\dots,-2u-\sum_{i=1}^r\delta_i)
\end{align*} 
and add the $\rho$-vector 
\begin{align*}
    (2u+\sum_{i=1}^r\delta_i-\frac{1}{2},...,2u+\sum_{i=2}^r\delta_i+\frac{1}{2},2u+\sum_{i=2}^r\delta_i-\frac{1}{2},...,\epsilon_1+\frac{1}{2},\epsilon_1-\frac{1}{2},...,\frac{1}{2}).
\end{align*}

We obtain the following vector 
\begin{align*}
    (\overbrace{2u+\sum_{i=1}^r\delta_i-\epsilon_1-\frac{1}{2},\dots,2u+\sum_{i=2}^r\delta_i+\frac{1}{2}-\epsilon_1}^{\delta_1\text{ times}},2u+\sum_{i=2}^r\delta_i-\frac{1}{2}-\epsilon_1-\epsilon_2,
\end{align*}    
    
\begin{align*}
    \dots,-2u-\sum_{i=2}^r\delta_i+\epsilon_1+\frac{1}{2},\overbrace{-2u-\sum_{i=1}^r\delta_i+\epsilon_1-\frac{1}{2},\dots,-2u-\sum_{i=1}^r\delta_i+\frac{1}{2}}^{\epsilon_1\text{ times}})
\end{align*}

The last $\epsilon_1$ terms are negative, so there are $\epsilon_1$ changes of signs to do. As $(2u+\sum_{i=1}^r\delta_i-\epsilon_1+\frac{1}{2})$ is strictly greater than $(2u+\sum_{i=1}^r\delta_i-\epsilon_1-\frac{1}{2})$, the last $\epsilon_1$ terms, when considered as positive, are all strictly greater than the first term. We also have that $(2u+\sum_{i=2}^r\delta_i-\epsilon_1-\frac{1}{2})$ is strictly less than $(2u+\sum_{i=2}^r\delta_i-\epsilon_1+\frac{1}{2})$ then the last $\epsilon_1+1$-th term, when considered as positive, is strictly less than the $\delta_1$-th term.

We then need to put the last $\epsilon_1$ terms of the vector in front of the first component, so we need to perform $(2u+\sum_{i=1}^r\delta_i-\epsilon_1)\times\epsilon_1$ transpositions, and this is the minimal number of transpositions required. Subsequently, we need to rearrange the first $\epsilon_1$ terms (which are increasing) to put them such that the first $\epsilon_1$ terms are decreasing, so there is $\frac{\epsilon_1(\epsilon_1-1)}{2}$ transpositions to do. In total, we have done $\frac{\epsilon_1\times(\epsilon_1+1)}{2}+\epsilon_1\times(2u+\sum_{i=1}^r\delta_i-\epsilon_1)$ operations contributing to the length at this step.

The next step uses $\frac{\epsilon_2\times(\epsilon_2+1)}{2}+\epsilon_2\times(2u+\sum_{i=2}^r\delta_i-\epsilon_1-\epsilon_2)$ operations, and the transposition used in this step will not interfere with the transpositions of the first step. So, we are still using the minimal number of transpositions required. Arguing by induction on $r$, we conclude that we are in case $(2)$ of the Borel-Weil-Bott theorem and that the length of the required $\sigma\in\W$ is $\sum_{i=1}^{r}\frac{\epsilon_i\times(\epsilon_i+1)}{2}+\sum_{i=1}^{r}\epsilon_i\times(2u+\sum_{j=i}^{r}\delta_j-\sum_{j=1}^i\epsilon_j)$.

In our notation, the sum of the components of $\alpha$ is equal to $\sum_{i=1}^{r}\epsilon_i\left(\sum_{j=i}^{r} \delta_j\right)$, the integer $2u$ is equal to $\sum_{i=1}^{r} \epsilon_i$.

We therefore have 
\begin{align*}
    l(\sigma)=\sum_{i=1}^{r}\frac{\epsilon_i(\epsilon_i+1)}{2}+\sum_{i=1}^{r}\epsilon_i\left(2u+\sum_{j=i}^{r}\delta_j-\sum_{j=1}^{i}\epsilon_j\right)
\end{align*}
\begin{align*}
    =\frac{\displaystyle\sum_{i=1}^{r}\epsilon_i^2}{2}+\frac{\displaystyle\sum_{i=1}^{r}\epsilon_i}{2}+|\alpha|+\sum_{i=1}^{r}\epsilon_i\left(\sum_{j=i+1}^{r}\epsilon_j\right)
\end{align*}
\begin{align*}
    =|\alpha|+u+\frac{\left(\displaystyle\sum_{i=1}^{r}\epsilon_i\right)^2}{2}
    =|\alpha|+u+2u^2.
\end{align*}
\end{proof}

The examples from the Borel-Weil-Bott section $1.3$ are in accordance with the lemma \ref{comb1}.

Knowing where the cohomology of the minimal free resolution of $\II'$ is concentrated, using the spectral sequence associated with hypercohomology, we can collect some information on the higher cohomology of $\II'$.
The next technical lemma gives not an explicit description of the length of the required element of $\W$ compared to lemma \ref{comb1}; however, the vanishing of the cohomology of the minimal free resolution of $\II'(4)$ is better located.

\begin{lemm}\label{comb2}
    The length of the required $\sigma\in\W$ in the Borel-Weil-Bott procedure applied to $K_{\lambda(\alpha,u)}\mathcal{F}(4)$ is strictly less than $|\alpha|+2u^2-u$.
\end{lemm}
\begin{proof}
    The partition $\lambda(\alpha,u)$ has the following form
    \begin{align*}              (\overbrace{2u+\sum_{i=1}^r\delta_i,\dots,2u+\sum_{i=1}^r\delta_i}^{\epsilon_1\text{ times}},2u+\sum_{i=2}^r\delta_i,\dots,\epsilon_1+\epsilon_2,\overbrace{\epsilon_1,\dots,\epsilon_1}^{\delta_1\text{ times}})
    \end{align*}
    The twist by $\OO_{\GR}(2)$ gives the following vector
    \begin{align*}
    (2u-2+\sum_{i=1}^r\delta_i,...,2u-2+\sum_{i=1}^r\delta_i,2u-2+\sum_{i=2}^r\delta_i,...,-2+\epsilon_1+\epsilon_2,-2+\epsilon_1,...,-2+\epsilon_1)
    \end{align*}

    We now apply the Borel-Weil-Bott procedure.
    The first step is to consider the following vector
    \begin{align*}
            (-\epsilon_1+2,\dots,-\epsilon_1+2,-\epsilon_1-\epsilon_2+2,\dots,-2u-\sum_{i=2}^r\delta_i+2,-2u-\sum_{i=1}^r\delta_i+2,\dots,-2u-\sum_{i=1}^r\delta_i+2)
    \end{align*}
    and add the $\rho$-vector
    \begin{align*}
    (2u+\sum_{i=1}^r\delta_i-\frac{1}{2},...,2u+\sum_{i=2}^r\delta_i+\frac{1}{2},2u+\sum_{i=2}^r\delta_i-\frac{1}{2},...,\epsilon_1+\frac{1}{2},\epsilon_1-\frac{1}{2},...,\frac{1}{2}).
    \end{align*}
    
    We obtain the following vector 
    \begin{align*}
        (\overbrace{2u+\sum_{i=1}^r\delta_i-\epsilon_1+\frac{3}{2},\dots,2u+\sum_{i=2}^r\delta_i+\frac{5}{2}-\epsilon_1}^{\delta_1\text{ times}},2u+\sum_{i=2}^r\delta_i+\frac{3}{2}-\epsilon_1-\epsilon_2,
    \end{align*}
    \begin{align*}
        \dots,-2u-\sum_{i=2}^r\delta_i+\epsilon_1+\frac{5}{2},\overbrace{-2u-\sum_{i=1}^r\delta_i+\epsilon_1+\frac{3}{2},\dots,-2u-\sum_{i=1}^r\delta_i+\frac{5}{2}}^{\epsilon_1\text{ times}}).
    \end{align*}
    The first component is $(2u+\sum_{i=1}^r\delta_i-\epsilon_1+\frac{3}{2})$, and the $\epsilon_1$-th component from the end is (up to the sign) $(2u+\sum_{i=1}^r\delta_i-\epsilon_1-\frac{3}{2})$. If two terms were equal, the cohomology would vanish. Hence, we may assume that $\epsilon_1+\lambda_1\leq4$.

    We now have the following inequalities
    \begin{align*}
    (2u+\sum_{i=2}^r\delta_i+\frac{5}{2}-\epsilon_1)\geq(2u+\sum_{i=1}^r\delta_i-\epsilon_1-\frac{3}{2}) \text{ because } 4\geq\delta_1+\epsilon_1.
    \end{align*}
    \begin{align*}
        (2u+\sum_{i=2}^r\delta_i+\frac{5}{2}-\epsilon_1)\geq-(2u+\sum_{i=1}^r\delta_i-\epsilon_1-\frac{3}{2}) \text{ because } 2u\geq\epsilon_1 \text{ and } \frac{5}{2}\geq\frac{3}{2}
    \end{align*}
    \begin{align*}
    (2u+\sum_{i=2}^r\delta_i+\frac{5}{2}-\epsilon_1)\geq(2u+\sum_{i=1}^r\delta_i-\frac{5}{2}) \text{ because } 5\geq\delta_1+\epsilon_1.
    \end{align*}
    \begin{align*}
        (2u+\sum_{i=2}^r\delta_i+\frac{5}{2}-\epsilon_1)\geq-(2u+\sum_{i=1}^r\delta_i-\frac{5}{2})\text{ because } 2u\geq\epsilon_1
    \end{align*}
    The first two inequalities compare the $\delta_1$-th component from the beginning with the $\epsilon_1$-th component from the end.
    The last two inequalities compare the $\delta_1$-th component from the beginning with the first component from the end.

    Compared to lemma \ref{comb1}, we can eliminate $\epsilon_1\times\delta_1$ operations in this step because the first $\delta_1$ terms will stay at the beginning of the vector. The length of the required $\sigma$ in this twisted situation is bounded by the length of the required element of $\W$ in the non-twisted situation minus all the contributions described above. The length satisfies $l(\sigma)\leq|\alpha|+2u^2+u-\sum_{i=1}^r\epsilon_i\times\delta_i$, which does not suffice to obtain the required inequality.

    We need to analyze the last step further to determine whether some operations can be omitted.
    Before the last step, in the middle, the components are
    \begin{align*}
    (\dots,\overbrace{\delta_r+\frac{3}{2},\dots,\frac{5}{2}}^{\delta_r\text{ times}},\overbrace{\frac{3}{2}-\delta_r,\dots,-\delta_r-\epsilon_r+\frac{5}{2}}^{\epsilon_r\text{ times}},\dots)    
    \end{align*}

    We will examine each case individually.
    
    If $\delta_r\geq2$ (or $\delta_i\geq2$), then $l(\sigma)\leq |\alpha|+2u^2+u-\sum_{i=1}^r\epsilon_i-1$, so we have that $l(\sigma)<|\alpha|+2u^2-u$ as required.
    
    If $\delta_r=1$, as the components in the middle are
    \begin{align*}
    (\dots,\frac{5}{2},\overbrace{\frac{1}{2},\dots,-\epsilon_r+\frac{3}{2}}^{\epsilon_r\text{ times}},\dots)    
    \end{align*}
    we see that the first term in the "$\epsilon_r\text{ times}$" bracket is already positive. We have one less operation to perform. The length satisfies $l(\sigma)\leq|\alpha|+2u^2+u-\sum_{i=1}^r\epsilon_i\times\delta_i-1<|\alpha|+2u^2-u$.

    If $\delta_r=0$, the components in the middle are
    \begin{align*}
    (\dots,\overbrace{\frac{3}{2},\dots,-\epsilon_r+\frac{5}{2}}^{\epsilon_r\text{ times}},\dots)    
    \end{align*}
    
    If $\epsilon_r\geq3$, then two terms are equal up to the sign, and we have no cohomology.
    
    If $\epsilon_r=2$, then we have zero operation to perform on the "$\epsilon_r\text{ times}$" bracket, so the length $l(\sigma)$ satisfies $l(\sigma)\leq|\alpha|+2u^2+u-\sum_{i=1}^{r-1}\epsilon_i\times\delta_i-3<|\alpha|+2u^2-u$.

    If $\epsilon_r=1$, we can also assume that $\delta_{r-1}=1$ (otherwise we are in the first case), then before doing the penultimate step, the component at the right of $\frac{3}{2}$ is $-\frac{1}{2}$. So we don't have to transpose these two components, and we have that the component of the "$\epsilon_r\text{ times}$" bracket is already positive. We can omit two operations, which gives us the required inequality on the length.

\end{proof}

The examples from the Borel-Weil-Bott section $1.3$ are in accordance with the lemma \ref{comb2}. The following two lemmas will be used in the proof of the theorem on the formula for the number of quadrics vanishing on the Kummer variety.

\begin{lemm}
    For all $p\geq1$, the cohomology group $\HH^p(\PP,\II(4))$ vanishes
\end{lemm}
\begin{proof}
    Consider the spectral sequence associated with the hypercohomology of the minimal free resolution $G_\bullet(4)$ of $\II'(4)$. The first page of the spectral sequence is 
    \begin{align*}
    E_1^{p,q}=\HH^q(\GR,G_{-p+1}(4))\Rightarrow\HH^{-p+q}(\GR,\II'(4)).
    \end{align*}
    The lemma follows from this spectral sequence: for the $i$-term of the minimal free resolution, any non-zero cohomology is concentrated in a degree strictly less than $i$.
\end{proof}

\begin{lemm}\label{dim}
    The space of global sections $\HH^0(\GR,\OO_{\GR}(4))$ has dimension $\frac{8^g+2^g}{2}$.
\end{lemm}
\begin{proof}
    The space of global sections $\HH^0(\GR,i'_*\OO_K(4))$ is equal to $\HH^0(K,\OO(8\Theta))$. The latter vector space is isomorphic to $\HH^0(J_C,\OO(8\Theta))^+$, the fixed point of $\HH^0(J_C,\OO(8\Theta))$ under the involution $[-1]$. We will compute the dimension of this vector space thanks to the Woods Hall trace formula \cite[Exposé III, Corollary 6.12]{SGA5} applied to $[-1]$ (the graph of $[-1]$ and the diagonal intersect transversally inside $J_C\times J_C$). We then obtain that
    \begin{align*}
        \operatorname{tr}([-1],\operatorname{H}^0(J_C,\OO(8\Theta))=\sum_{x\in J_C[2]} \frac{\operatorname{tr}([-1],x^*\OO(8\Theta))}{\operatorname{det(1-d[-1],x^*\Omega_{A/\CC}^1)}}.
    \end{align*}
    
    As $[-1]$ is an involution, it acts by $1$ or $-1$ on the stalk of the line bundle $\OO(\Theta)$, hence it will act by $1$ on the stalk of the line bundle $\OO(8\Theta)=\OO(\Theta)^{\otimes8}$. The differential of $[-1]$ is $-\operatorname{id}$, as we have $4^g$ points of $2$-torsion; the right-hand side of the Woods Hall trace formula is then $2^g$.
    The left-hand side is equal to $\operatorname{dim}\operatorname{H}^0(J_C,\OO(8\Theta)^+-\operatorname{dim}\operatorname{H}^0(J_C,\OO(8\Theta))^-$. On the other hand, we have that $8^g=\operatorname{dim}\operatorname{H}^0(J_C,\OO(8\Theta))=\operatorname{dim}\operatorname{H}^0(J_C,\OO(8\Theta)^++\operatorname{dim}\operatorname{H}^0(J_C,\OO(8\Theta))^-$. The dimension of $\operatorname{H}^0(J_C,\OO(8\Theta)^+$ is then $\frac{8^g+2^g}{2}$.    
\end{proof}
\begin{theo}
    There are
    \begin{align*}
        \binom{2^g+3}{4}-\frac{8^g+2^g}{2}
    \end{align*} linearly independent quartics defining the Kummer variety and 
    \begin{align*}
        \frac{2^g}{2g+1}\binom{2g+1}{g}\prod_{i < j}\frac{2g+4-i-j}{2g+2-i-j}-\frac{8^g+2^g}{2}
    \end{align*} linearly independent quartics modulo quartics arising from quadrics.
\end{theo}

\begin{proof}
    The theorem follows from these two long exact sequences associated with these short exact sequences as we will explain:
    \begin{align*}
        0\to\II'(4)\to\OO_{\GR}(4)\to i'_*\OO_K(4)\to0
    \end{align*}
    \begin{align*}
        0\to\II(4)\to\OO_{\PP}(4)\to i_*\OO_K(4)\to0
    \end{align*}
    
    As $\OO_{\GR}(4)$ is the Weyl functor associated with $(-2,\dots,-2)$, the required element of $\W$ is the neutral element of $\W$.    
    The space of global sections $\HH^0(\GR,\OO_{\GR}(4))$ is $V_{(2,\dots,2)}$ which, by the Weyl dimension formula \cite[Section 24.3]{Hum72}, has dimension $\frac{2^g}{2g+1}\binom{2g+1}{g}\prod_{1\leq i < j\leq g}\frac{2g+4-i-j}{2g+2-i-j}$.
    The space  of global sections $\HH^0(\GR,\OO_{\GR}(4))$ has dimension $\frac{8^g+2^g}{2}$ thanks to lemma \ref{dim}.
    
    The vanishing of $\HH^1(\GR,\II'(4))$ gives the surjectivity of the following map
    \begin{align*}
        \HH^0(\GR,\OO_{\GR}(4))\to\HH^0(K,\OO_K(4)).
    \end{align*}
    
    Then the space of global sections $\HH^0(\GR,\II'(4))$ has dimension
    \begin{align*}
        \frac{2^g}{2g+1}\binom{2g+1}{g}\prod_{1\leq i < j\leq g}\frac{2g+4-i-j}{2g+2-i-j}-\frac{8^g+2^g}{2}.
    \end{align*}
    
    As the map $\HH^0(\PP,\OO_{\PP}(4))\to\HH^0(\GR,\OO_{\GR}(4))$ is also surjective, the map $\HH^0(\PP,\OO_{\PP}(4))\to\HH^0(K,\OO_K(4))$ is then surjective. The space of global sections of $\OO_{\PP}(4)$ has dimension $\binom{2^g+3}{4}$. 
    
    Then the space of global sections $\HH^0(\PP,\II(4))$ has dimension 
    \begin{align*}
        \binom{2^g+3}{4}-\frac{8^g+2^g}{2}.
    \end{align*}

\end{proof}

\end{document}